\documentclass[12pt]{article}
\usepackage{amsmath,amssymb,amsthm}

\begin{document}
\newtheorem{thm}{Theorem}[section]
\newtheorem{pro}[thm]{Proposition}
\newtheorem{cor}[thm]{Corollary}
\newtheorem{lem}[thm]{Lemma}
\newtheorem{dfn}[thm]{Definition}
\newtheorem{rem}[thm]{Remark}
\newtheorem{conj}[thm]{Conjecture}
\title{Tight Lagrangian surfaces in $S^{2} \times S^{2}$
}
\author{Hiroshi Iriyeh and Takashi Sakai
\footnote{The second author was partially supported
by Grant-in-Aid for Young Scientists (B) No. 20740044,
The Ministry of Education, Culture, Sports, Science and Technology, Japan.}}
\date{}
\maketitle

\begin{abstract}
We determine all tight Lagrangian surfaces in $S^2 \times S^2$.
In particular, globally tight Lagrangian surfaces in $S^2 \times S^2$
are nothing but real forms.


\end{abstract}

\section{Introduction and main results}

In 1991, Y.-G.\ Oh \cite{Oh91} introduced the notion of {\it tightness} of
closed Lagrangian submanifolds in compact Hermitian symmetric spaces.
Let $(\tilde M=G/K,\omega,J)$ be a Hermitian symmetric space of compact type
and $L$ be a closed embedded Lagrangian submanifold of $\tilde M$.
Then $L$ is said to be {\it globally tight} (resp. {\it tight}) if it
satisfies
\begin{eqnarray*}
 \#(L \cap g \cdot L)=SB(L,\mathbb{Z}_{2})
\end{eqnarray*}
for any isometry $g \in G$ (resp. close to the identity) such that
$L$ transversely intersects with $g \cdot L$.
Here $SB(L,\mathbb{Z}_{2})$ denotes the sum of $\mathbb{Z}_{2}$-Betti numbers
of $L$.

It is known that any real forms in a compact Hermitian symmetric space $G/K$
are tight.
It is a natural problem to classify all tight Lagrangian submanifolds in $G/K$.
Indeed, Oh \cite{Oh91} proved the following uniqueness theorem in 
$\mathbb{C}P^n$.
\begin{thm}[Oh] \label{thm:1-1}
Let $L$ be a closed embedded tight Lagrangian submanifold in $\mathbb{C}P^n$.
Then $L$ is the standard totally geodesic $\mathbb{R}P^n$ if $n \geq 2$
or it is the standard embedding $S^1 (\cong \mathbb{R}P^1)$ into
$S^2 (\cong \mathbb{C}P^1)$ as a latitude circle.
\end{thm}
And he posed the following problem:

\bigskip

\noindent
{\bf Problem (Oh).}\ Classify all possible tight Lagrangian submanifolds
in other Hermitian symmetric spaces.
Are the real forms on them the only possible tight Lagrangian submanifolds?

\bigskip

In this paper we give the complete solution of it in the case of
$S^2 \times S^2$.
Note that the following is the first result for the above problem except
the case of $\mathbb{C}P^n$.
\begin{thm} \label{thm:1-2}
Let $L$ be a closed embedded tight Lagrangian surface in 
$(S^2 \times S^2,\omega_{0} \oplus \omega_{0})$, where $\omega_{0}$ denotes
the standard K\"{a}hler form of $S^2(1) \cong {\mathbb C}P^1$.
Then $L$ must be one of the following cases:\newline
$\mathrm{(i)}$\ the totally geodesic Lagrangian sphere
$$ L= \{ (x, -x) \in S^2 \times S^2 \ | \ x \in S^2 \}. $$
$\mathrm{(ii)}$\ a product of latitude circles $S^1(a) \subset S^2$, i.e.,
$$ L=S^1(a) \times S^1(b) \subset S^2 \times S^2, $$
where $S^1(a)$ stands for the round circle with radius $a \ (0 < a \leq 1)$. 
\end{thm}

\begin{cor} \label{cor:1-3}
Let $L$ be a closed embedded globally tight Lagrangian surface in 
$(S^2 \times S^2,\omega_{0} \oplus \omega_{0})$.
Then $L$ must be one of the following two cases:\newline
$\mathrm{(i)}$\ the totally geodesic Lagrangian sphere
$$ L= \{ (x, -x) \in S^2 \times S^2 \ | \ x \in S^2 \}. $$
$\mathrm{(ii)}$\ the product of equators (totally geodesic Lagrangian torus)
$$ L=S^1(1) \times S^1(1) \subset S^2 \times S^2. $$
\end{cor}

As Oh\cite[p.\ 409]{Oh91} pointed out, the global tightness is closely
related with the Hamiltonian volume minimization problem.
In fact, all globally tight Lagrangian submanifolds which are listed
in Theorem \ref{thm:1-1} and Corollary \ref{cor:1-3} are Hamiltonian volume
minimizing\footnote{The Lagrangian surface $\mathrm{(i)}$ in
Corollary \ref{cor:1-3} is actually homologically volume minimizing.}
(see \cite{Oh90,IOS}).

Our strategy of the proof of the main result (Theorem \ref{thm:1-2}) is to
classify all tight Lagrangian surfaces by their Killing nullities in
$S^2 \times S^2$.
Possible Killing nullities of Lagrangian surfaces in $S^2 \times S^2$
are $3,4,5$ and $6$.
In Section 3, we shall show that it is impossible for a tight Lagrangian
surface $L$ to have $6$ or $5$ as the Killing nullity using the theory of
tight maps into Euclidean spaces. This part is a modification of Oh's
method used in the case of $\mathbb{C}P^n$ (see \cite[Theorem 4.4]{Oh91}).
But, in our case, we essentially use the equality condition of Kuiper's
inequality (see Theorem \ref{thm:Kuiper}) in the case where the Killing
nullity of $L$ is $5$ (Proposition \ref{pro:3-4}).

The latter part of the paper is devoted to the determination of Lagrangian
surfaces in $S^2 \times S^2$ with low Killing nullities.
In Section 4, first of all, we explain basic inequality obtained by
Gotoh \cite{Gotoh99}, which gives a lower bound of the Killing
nullity of any submanifold in compact symmetric
spaces (see Theorem \ref{thm:Killing nullity}).
In Section 5, we shall prove a sharp estimate of the lower bound in Gotoh's
inequality in the case of Lagrangian surfaces in
$S^2 \times S^2$ (Proposition \ref{pro:inequality of the Killing nullity}).
This formula enables us to determine all the Lagrangian surfaces with
low Killing nullities.
In the last section, all the Lagrangian surfaces with Killing nullities
3 or 4 are completely determined.
Our argument is based on Gotoh's inequality, the above mentioned estimate
and recent developments concerning Lagrangian
surfaces in $S^2 \times S^2$ (see \cite{CU, MO}).
In particular, Gotoh's inequality is used effectively in this context.

\section{Preliminaries}
\setcounter{equation}{0}

Let $(\tilde M,\omega)$ be a closed symplectic manifold.
Let $L$ be a manifold of dimension $\frac{1}{2}\mathrm{dim}\tilde M$.
In this paper all manifolds, maps, etc.\ are supposed to be of class
$C^{\infty}$.
An embedding $\iota:L \to \tilde M$ is said to be {\it Lagrangian}
if $\iota^{*}\omega = 0$.
The image $\iota(L)$ is called (embedded) {\it Lagrangian submanifold} of
$\tilde M$.
Wherever possible, we denote $\iota(L)$ by $L$.
In this paper, we only consider a special class of symplectic manifolds,
i.e., K\"{a}hler manifolds.
If $J$ is the associated complex structure on $(\tilde M,\omega)$,
then the metric $g$ and $\omega$ have the relation $g(X,Y)=\omega(X,JY)$.
Then $\iota:L \to \tilde M$ is Lagrangian if and only if
\begin{eqnarray*}
 T_{\iota(p)}\tilde M = \iota_{*}T_{p}L \oplus J(\iota_{*}T_{p}L)
\end{eqnarray*}
for any $p \in L$ as an orthogonal direct sum.

\bigskip

Let us introduce the notion of tightness of Lagrangian submanifolds.
Although Oh considered the case of Hermitian symmetric spaces in \cite{Oh91},
its definition is valid for, more generally, homogeneous K\"{a}hler manifolds.
\begin{dfn} \label{dfn:tightness} \rm
Let $(\tilde M,\omega,J)$ be a homogeneous K\"{a}hler manifold and $L$ be a
Lagrangian submanifold of $\tilde M$.
Then $L$ is said to be {\it globally tight} (resp. {\it tight}) if
 \begin{eqnarray*}
  \#(L \cap g \cdot L)=SB(L,\mathbb{Z}_{2})
 \end{eqnarray*}
for any holomorphic isometry $g$ (resp. close to the identity) such that $L$
transversely
intersects with $g \cdot L$.
\end{dfn}
One of the important tools to study the tightness of Lagrangian submanifolds
is the theory of tight maps into Euclidean spaces.
We recall some necessary definitions and results for our discussion
in the following sections.

\bigskip

Let $M^n$ be a closed $n$-dimensional manifold.
\begin{dfn} \rm
A nondegenerate function $f$ on $M^n$ is said to be {\it tight}
(or {\it perfect}) if it has the minimal number of critical points:
 \begin{eqnarray*}
  \#\mathrm{Crit}(f)=SB(M,\mathbb{Z}_{2}),
 \end{eqnarray*}
where $\# \mathrm{Crit}(f)$ denotes the number of critical points of $f$.
\end{dfn}
\begin{rem} \rm
By Morse theory, for any nondegenerate function $f \in C^{\infty}(M^n)$
we have
 \begin{eqnarray*}
  \#\mathrm{Crit}(f) \geq SB(M,\mathbb{Z}_{2}).
 \end{eqnarray*}
\end{rem}

\begin{dfn} \rm
A map $\phi:M^n \to \mathbb{E}^N$ from a closed manifold $M^n$ to the
$N$-dimensional Euclidean space $(\mathbb{E}^N, \langle \cdot,\cdot \rangle)$
with the standard inner metric $\langle \cdot,\cdot \rangle$ is said to be
{\it tight} if the functions $z \circ \phi$
are tight for all unit vectors $z^* \in S^{N-1}$ such that $z \circ \phi$
is nondegenerate, where $z$ is the linear function dual to $z^*$:
 \begin{eqnarray*}
  (z \circ \phi)(x) := \langle \phi(x), z^* \rangle \qquad (x \in M^n).
 \end{eqnarray*}
\end{dfn}
\begin{rem} \rm
Sard's theorem says that the function $z \circ \phi$ is nondegenerate for
almost all $z^* \in S^{N-1}$.
\end{rem}

\begin{dfn} \rm
A map $\phi:M^n \to \mathbb{E}^N$ is said to be {\it substantial}
(or {\it full}) if the image of $\phi$ is not contained in any hyperplane of
$\mathbb{E}^N$.
\end{dfn}

The following inequality by Kuiper \cite[Theorem 3A]{Kuiper70} will be used 
to prove the nonexistence result of tight Lagrangian surfaces in
$S^2 \times S^2$ with Killing nullity $6$ or $5$.

\begin{thm}[Kuiper \cite{Kuiper70,Kuiper62}, Little-Pohl \cite{LP}]
 \label{thm:Kuiper}
Let $M^n$ be a closed $n$-dimensional manifold.
If $\phi:M^n \to \mathbb{E}^N$ is a tight smooth map substantially into
$\mathbb{E}^N$, then
 \begin{eqnarray} \label{eq:2-1}
  N \leq \frac{1}{2} n(n+3).
 \end{eqnarray}
Moreover, the equality is only obtained if $M=\mathbb{R}P^n$,
the $n$-dimensional real projective space and the image $\phi(M)$ is
the Veronese manifold (unique up to projective transformation) of
$\mathbb{E}^N$.
\end{thm}

Note that the equality condition above was obtained by Kuiper \cite{Kuiper62}
for surfaces, $n=2$, and by Little and Pohl \cite{LP}
for $n$-manifolds in general.
We will use the equality condition for the case of surfaces essentially
in Section 3.

\bigskip

At the end of this section, we review the definition of the Killing nullity.
Let $\tilde M$ be a Riemannian manifold and $M$ be a submanifold in $\tilde M$.
Let $\mathfrak{i}(\tilde M)$ be the Lie algebra consisting of all Killing
vector fields of $\tilde M$. Consider the following vector space
 \begin{eqnarray*}
  \mathfrak{i}(\tilde M)^{NM}
  := \{ Z^{NM} \in \Gamma(NM) \ | \ Z \in \mathfrak{i}(\tilde M) \},
 \end{eqnarray*}
where $NM$ denotes the normal bundle of $M$ and $Z^{NM}$ indicates the normal
component of a vector field $Z$ of $\tilde M$.
The dimension of $\mathfrak{i}(\tilde M)^{NM}$ is called the
{\it Killing nullity} of $M$ and denoted by $\mathrm{nul}_K(M)$.
For any $p \in M$, we consider a linear map
\begin{eqnarray*}
 \Phi_p : \mathfrak{i}(\tilde M)^{NM} \longrightarrow
  N_{p}M \oplus \mathrm{Hom}(T_{p}M,N_{p}M)
\end{eqnarray*}
defined by
\begin{eqnarray*}
 \Phi_p(Z^{NM}) := (Z_{p}^{NM}, \nabla^{NM} Z^{NM}),
\end{eqnarray*}
where $\nabla^{NM}$ denotes the normal connection of the normal bundle $NM$.
By definition of $\Phi_p$, we have
\begin{eqnarray*}
  \mathrm{nul}_K(M) \geq \dim \mathrm{Im}\Phi_p.
\end{eqnarray*}
As we explain in Section 4, this estimate can be described in terms of
Lie algebra in the case where $\tilde M$ is a compact Riemannian symmetric
space.

\section{Nonexistence of tight Lagrangian surfaces with large Killing nullities}
\setcounter{equation}{0}

Let $G$ be the identity component of the full isometry group of
$S^2 \times S^2$, that is, $G=SO(3) \times SO(3)$.
Then the isotropy group $K$ at $o = (p_1, p_2)$ in $S^2 \times S^2$
is isomorphic to $SO(2) \times SO(2)$,
and $S^2 \times S^2$ is expressed as a coset space $G/K$.
Assume that $G$ is equipped with an invariant metric normalized so that
$G/K$ becomes isometric to the product of unit spheres.

The vector space of all Killing vector fields on $S^2 \times S^2$ is
denoted by $\mathfrak{i}(G/K)$, which is isomorphic to the Lie algebra
$\mathfrak{g}=\mathfrak{so}(3) \oplus \mathfrak{so}(3)$ of $G$.
Let $\iota : L \to S^2 \times S^2$ be a Lagrangian embedding of a closed
surface $L$.
Let us consider the Killing nullity of $L$:
$\mathrm{nul}_K(L)=\dim_{\mathbb R} \mathfrak{i}(G/K)^{NL}$.
Since the dimension of $G$ is $6$, we have $\mathrm{nul}_K(L) \leq 6$.

\begin{pro} \label{pro:3-1}
If $\mathrm{nul}_K(L)=6$, then the closed embedded Lagrangian surface $L$ in
$S^2 \times S^2$ cannot be tight in the sense of
Definition \ref{dfn:tightness}.
\end{pro}
{\it Proof.} Assume that $\mathrm{nul}_K(L)=6$.
It suffices to show the following:

{\bf Claim 1}.\
There exists some $g \in G$ which is arbitrarily close to the identity
such that $L$ transversely intersects with $g \cdot L$ and
 \begin{eqnarray*}
  \#(L \cap g \cdot L) > SB(L,\mathbb{Z}_2).
 \end{eqnarray*}

It is equivalent to

{\bf Claim 2}.\
There exists some $W \in \mathfrak{i}(G/K)$ such that all zeros of
$W^{NL} \in \Gamma(NL)$ are nondegenerate and
 \begin{eqnarray*}
  \# \mathrm{Zero}(W^{NL}) > SB(L,\mathbb{Z}_2).
 \end{eqnarray*}

Indeed, for $W \in \mathfrak{i}(G/K)$ as in Claim 2, $\exp(tW) \in G$ will
be an isometry satisfying the property of Claim 1 for sufficiently small $t$.
Hence, we shall prove Claim 2.

Choose $W_1, W_2, \ldots, W_6 \in \mathfrak{i}(G/K)$ such that
$\{ W_1^{NL}, W_2^{NL}, \ldots, W_6^{NL} \}$ form a basis of
$\mathfrak{i}(G/K)^{NL}$.
Since the isometry group action of $G$ on $G/K$ is Hamiltonian,
there exists a function $f_\xi \in C^\infty(G/K)$ corresponding to any element
$\xi \in \mathfrak{i}(G/K) \cong \mathfrak{so}(3) \oplus \mathfrak{so}(3)$
such that
\begin{eqnarray*}
 df_\xi = \omega(\xi, \cdot).
\end{eqnarray*}
Therefore, for each $W_i \in \mathfrak{i}(G/K) \ (i = 1, 2, \ldots, 6)$,
there exists a function $f_i \in C^\infty(G/K)$ such that
$df_i = \omega(W_i, \cdot)$.

Let us define $\phi_i := \iota^{*}f_i \in C^\infty(L) \ (i = 1, 2, \ldots, 6)$
and introduce a smooth map $\phi : L \to \mathbb{E}^6$ defined by
\begin{eqnarray*}
 \phi(p) := (\phi_{1}(p), \phi_{2}(p), \ldots, \phi_{6}(p)) \quad (p \in L).
\end{eqnarray*}
We shall show that $\phi : L \to \mathbb{E}^6$ is substantial.
Assume that $\phi$ is not substantial, i.e., there exists a hyperplane
$H \subset \mathbb{E}^6$ such that $\phi(L) \subset H$.
This condition is equivalent to the one that there exists a nonzero vector
$z^* = (b_1, b_2, \ldots, b_6) \in \mathbb{R}^6$ such that
\begin{eqnarray*}
 (z \circ \phi)(p) = \langle z^*, \phi(p) \rangle = \mathrm{const.}
\end{eqnarray*}
for all $p \in L$.
Hence the differential of $z \circ \phi$ vanishes identically and we have
\begin{eqnarray*}
 0 = d(z \circ \phi)
 &=& d(b_1 \phi_1 + b_2 \phi_2 + \cdots + b_6 \phi_6) \nonumber\\
 &=& b_1 d(\iota^{*}f_1) + b_2 d(\iota^{*} f_2)
  + \cdots + b_6 d(\iota^{*}f_6) \nonumber\\
 &=& \iota^{*}(b_1 df_1 + b_2 df_2 + \cdots + b_6 df_6) \nonumber\\
 &=& \iota^{*}(b_1 \omega(W_1, \cdot) + b_2 \omega(W_2, \cdot)
  + \cdots + b_6 \omega(W_6, \cdot)) \nonumber\\
 &=& \omega(b_1 W_1 + b_2 W_2 + \cdots + b_6 W_6, \iota_{*}(\cdot)).
\end{eqnarray*}
Putting $V := b_1 W_1 + b_2 W_2 + \cdots + b_6 W_6 \in \mathfrak{i}(G/K)$,
then we obtain
\begin{eqnarray*}
 0 = \omega(V^{NL}, \iota_{*}(\cdot)) = g(J V^{NL}, \iota_{*}(\cdot))
\end{eqnarray*}
and it yields
$V^{NL} = b_1 W_1^{NL} + b_2 W_2^{NL} + \cdots + b_6 W_6^{NL} =0$.
Since $W_1^{NL}, W_2^{NL}, \ldots, W_6^{NL}$ are linearly independent,
we have $z^* = (b_1, b_2, \ldots, b_6) = 0$.
This is a contradiction.
Therefore, $\phi : L \to \mathbb{E}^6$ is substantial.

By Theorem \ref{thm:Kuiper}, the smooth map $\phi$ cannot be tight.
Hence, there is some
$z^* =(a_1, a_2, \ldots, a_6) \in S^5(1) \subset \mathbb{E}^6$
such that 
\begin{eqnarray*}
 z \circ \phi 
  = a_1 \phi_1 + a_2 \phi_2 + \cdots + a_6 \phi_6 \in C^{\infty}(L)
\end{eqnarray*}
is nondegenerate but not tight, that is,
\begin{eqnarray*}
 \# \mathrm{Crit}(z \circ \phi) > SB(L, \mathbb Z_2).
\end{eqnarray*}
The differential of $z \circ \phi$ is calculated as
\begin{eqnarray*}
 d(z \circ \phi)
 = \omega(a_1 W_1 + a_2 W_2 + \cdots + a_6 W_6, \iota_{*}(\cdot)).
\end{eqnarray*}
Here, if we put 
\begin{eqnarray} \label{eq:3-1}
 W := a_1 W_1 + a_2 W_2 + \cdots + a_6 W_6 \in \mathfrak{i}(G/K),
\end{eqnarray}
then the critical points of $z \circ \phi$ coincide with the zeros of
$W^{NL} \in \mathfrak{i}(G/K)^{NL}$.
The Killing vector field $W$ in (\ref{eq:3-1}) satisfies the requirement of
Claim 2:
\begin{eqnarray*}
 \# \mathrm{Zero}(W^{NL}) = \# \mathrm{Crit}(z \circ \phi)
  > SB(L, \mathbb Z_2).
\end{eqnarray*}
\hfill \qed

\begin{rem} \rm
The above proposition can be generalized to the case of complex
hyperquadrics $Q_n(\mathbb{C})$.
It is expressed as a coset space $G/K = SO(n+2)/SO(2) \times SO(n)$.
Let $L$ be a closed Lagrangian submanifold in $G/K$.
Since the dimension of the Lie algebra of $G=SO(n+2)$ is $(n+2)(n+1)/2$,
we have $\mathrm{nul}_K(L) \leq (n+2)(n+1)/2$.
The same argument of Proposition \ref{pro:3-1} implies that {\it if
$\mathrm{nul}_K(L) = (n+2)(n+1)/2$, then the closed embedded Lagrangian
submanifold $L$ in $Q_n(\mathbb{C})$ cannot be tight in the sense of
Definition \ref{dfn:tightness}}.
\end{rem}

Before we proceed to the case where
$\mathrm{nul}_K(L)=5$,
let us mention a topological restriction for embedded Lagrangian surfaces
of $S^2 \times S^2$.

\begin{lem} \label{lem:3-3}
Let $\iota : L \to S^2 \times S^2$ be a Lagrangian embedding of a closed
surface. Then the Euler characteristic $\chi(L)$ of $L$ is even.
\end{lem}
{\it Proof.} Let $\iota : L \to S^2 \times S^2$ be a Lagrangian embedding.
Then $\iota_{*}[L]$ defines an element of $2$-dimensional integral homology
class $H_2(S^2 \times S^2,\mathbb{Z})$.
Since the homology class is generated by $S:=[S^2 \times \{ \mathrm{pt} \}]$
and $T:=[\{ \mathrm{pt} \} \times S^2]$,
the element $\iota_{*}[L]$ is represented as
\begin{eqnarray*}
 \iota_{*}[L]=mS+nT
\end{eqnarray*}
for some $m, n \in \mathbb{Z}$.
The self-intersection index of the cycle
$\iota_{*}[L] \in H_2(S^2 \times S^2,\mathbb{Z})$
is calculated as
\begin{eqnarray*}
 \iota_{*}[L] \cdot \iota_{*}[L] = (mS+nT) \cdot (mS+nT) = 2mn.
\end{eqnarray*}
This fact and Arnold's formula (see \cite[p. 200]{Gi})
\begin{eqnarray*}
 \iota_{*}[L] \cdot \iota_{*}[L] = \chi(L)
\end{eqnarray*}
(in the nonorientable case the equality is modulo $2$)
implies that $\chi(L)$ is even.
\hfill \qed

\bigskip

Let us consider the case where $\mathrm{nul}_K(L)=5$.

\begin{pro} \label{pro:3-4}
If $\mathrm{nul}_K(L)=5$, then the closed embedded Lagrangian surface $L$ in
$S^2 \times S^2$ cannot be tight in the sense of
Definition \ref{dfn:tightness}.
\end{pro}
{\it Proof.} Let $\iota : L \to S^2 \times S^2$ be a Lagrangian embedding
of a closed surface $L$.
Suppose that $L$ is tight, i.e., it satisfies
\begin{eqnarray*}
 \#(L \cap g \cdot L) = SB(L,\mathbb{Z}_2)
\end{eqnarray*}
for any isometry $g \in G$ close to the identity such that $L$ transversely
intersects with $g \cdot L$.
This condition implies that
\begin{eqnarray} \label{eq:3-2}
 \# \mathrm{Zero}(W^{NL}) = SB(L, \mathbb Z_2)
\end{eqnarray}
for any $W \in \mathfrak{i}(G/K)$ such that all zeros of
$W^{NL} \in \mathfrak{i}(G/K)^{NL}$ are nondegenerate.

Since $\mathrm{nul}_K(L)=5$, we can choose
$W_1, \ldots, W_5 \in \mathfrak{i}(G/K)$ such that
$\{ W_1^{NL}, \ldots ,W_5^{NL} \}$ form a basis of $\mathfrak{i}(G/K)^{NL}$.
For any $W_i \in \mathfrak{i}(G/K) \ (i = 1, \ldots, 5)$,
there exists a function
$f_i \in C^{\infty}(G/K)$ such that $df_i = \omega(W_i,\cdot)$.

Define $\phi_i := \iota^{*}f_i \in C^{\infty}(L) \ (i = 1, \ldots, 5)$
and consider a smooth map $\phi : L \to \mathbb{E}^5$ defined by
\begin{eqnarray*}
 \phi(p) := (\phi_{1}(p), \ldots, \phi_{5}(p)) \quad (p \in L).
\end{eqnarray*}
As in the proof of Proposition \ref{pro:3-1},
we see that $\phi$ is a substantial map.

For any $z^* =(a_1, \ldots, a_5) \in S^{4}(1) \subset \mathbb{E}^5$ such that
$z \circ \phi = a_1 \phi_1 + \cdots + a_5 \phi_5$ is nondegenerate, we obtain
\begin{eqnarray*}
 d(z \circ \phi)
  &=& \omega(a_1 W_1 + \cdots + a_5 W_5, \iota_{*}(\cdot)).
\end{eqnarray*}
Putting $W:= a_1 W_1 + \cdots + a_5 W_5 \in \mathfrak{i}(G/K)$, then we have
\begin{eqnarray} \label{eq:3-3}
  d(z \circ \phi)(p) = \omega\left(W^{NL}(p), \iota_{*_p}(\cdot) \right)
\end{eqnarray}
for any $p \in L$. By equations (\ref{eq:3-2}) and (\ref{eq:3-3}), we obtain
\begin{eqnarray*}
 \# \mathrm{Crit}(z \circ \phi) = \# \mathrm{Zero}(W^{NL})
  = SB(L, \mathbb Z_2)
\end{eqnarray*}
for any $z^* =(a_1, \ldots, a_5) \in S^{4}(1) \subset \mathbb{E}^5$ such that
$z \circ \phi$ is nondegenerate.
Hence, $\phi : L \to \mathbb{E}^5$ is a tight substantial map of a closed
surface.

Theorem \ref{thm:Kuiper} implies that $\phi$ satisfies the equality of
(\ref{eq:2-1}).
Hence, we have $L=\mathbb{R}P^2$ and $\phi$ is the Veronese embedding.
But Lemma \ref{lem:3-3} shows that $\mathbb{R}P^2$ cannot be realized as
a Lagrangian embedding into $S^2 \times S^2$, since $\chi(\mathbb{R}P^2)=1$.
\hfill \qed

\section{Symmetric spaces and Gotoh's inequality}
\setcounter{equation}{0}

In this section, we explain Gotoh's inequality which gives a lower bound
of the Killing nullity of any submanifold in compact symmetric spaces.

Let $(G, H)$ be a Riemannian symmetric pair and
$\mathfrak g = \mathfrak h + \tilde{\mathfrak m}$
its canonical decomposition, where $\mathfrak g$ and $\mathfrak h$ denote
the Lie algebras of $G$ and $H$, respectively,
and $\tilde{\mathfrak m}$ is naturally identified with the tangent space
$T_o(G/H)$ of the origin $o=H$ in the symmetric space $G/H$.
Let $M$ be a compact submanifold in $G/H$.
We may assume that $M$ contains the origin $o=H$.
Then $\tilde{\mathfrak m}$ is orthogonally decomposed as
$$
\tilde{\mathfrak m} = \mathfrak m + {\mathfrak m}^\perp,
$$
where subspaces $\mathfrak m$ and ${\mathfrak m}^\perp$ correspond to
the tangent space $T_oM$ of $M$ and the normal space $N_oM$,
respectively.
Hence, we have an orthogonal decomposition
$$
\mathfrak g = \mathfrak h + \mathfrak m + {\mathfrak m}^\perp.
$$
Any $Z \in \mathfrak g$ can be decomposed as
$$
Z = Z_{\mathfrak h} + Z_{\mathfrak m} + Z^\perp
$$
according to the above orthogonal decomposition.
Define two linear mappings
$\Psi_1 : \mathfrak g \rightarrow {\mathfrak m}^\perp$
and
$\Psi_2 : \mathfrak g \rightarrow \mathrm{Hom}(\mathfrak m, {\mathfrak m}^\perp)$
by
\begin{eqnarray*}
\Psi_1(Z) := Z^\perp \ \ \mathrm{and} \ \ 
\Psi_2(Z)(X) := (\mathrm{ad}_{\mathfrak g}(Z_{\mathfrak h})X)^\perp
- B(X, Z_{\mathfrak m}) \qquad (X \in \mathfrak m),
\end{eqnarray*}
where $B : \mathfrak m \times \mathfrak m \to {\mathfrak m}^\perp$
is the bilinear mapping corresponding to the second fundamental form
of $M$ at $o$.
Then the following theorem has been proven by T.\ Gotoh.

\begin{thm}[Gotoh \cite{Gotoh99}] \label{thm:Killing nullity}
Let $G/K$ be a compact Riemannian symmetric space and $M$
a compact connected submanifold of $G/K$.
Then, the Killing nullity of $M$ satisfies the inequality
\begin{eqnarray} \label{eq:4-1}
 \mathrm{nul}_K(M)
  \geq \mathrm{codim}(M) + \dim \mathrm{Im} (\Psi_2|_{\mathfrak h}).
\end{eqnarray}
Moreover, if $M$ satisfies the equality in (\ref{eq:4-1}),
then $M$ is an orbit of a closed subgroup of $G$,
i.e., $M$ is a homogeneous submanifold of $G/K$.
\end{thm}

\begin{rem} \rm
We note that the Killing nullity $\mathrm{nul}_K(M)$
is a global invariant of $M$.
On the other hand, the right hand side of (\ref{eq:4-1})
is determined at the origin $o \in M$,
because $\mathrm{Im} (\Psi_2|_{\mathfrak h})$ is only depend on
the choice of a subspace $\mathfrak m$ in $\tilde{\mathfrak m}$.
\end{rem}

\section{An estimate for the case of $S^2 \times S^2$}
\setcounter{equation}{0}

In this section we shall give an estimate of (\ref{eq:4-1})
in the case of $S^2 \times S^2$ explicitly.

We set
$$
S^2 \times S^2 := \{ (x, y) \in \mathbb R^3 \times \mathbb R^3
\ | \ \| x \| = \| y \| = 1 \}
\subset \mathbb R^3 \times \mathbb R^3,
$$
and assume that $S^2 \times S^2$ is equipped with
a complex structure $J := J_0 \oplus J_0$,
where $J_0$ is the canonical complex structure of $S^2$.
Let $\{ e_1, e_2, e_3 \}$ be the standard basis of $\mathbb R^3$
and put
$o := (e_1, e_1) \in S^2 \times S^2 \subset \mathbb R^3 \times \mathbb R^3$.
Then $J$ acts on the basis $\{ (e_2, 0), (e_3, 0), (0, e_2), (0,e_3) \}$
of the tangent space $T_o(S^2 \times S^2)$ of $S^2 \times S^2$ at the
origin $o$ as:
\begin{eqnarray*}
J(e_2, 0) = (e_3, 0),
& & J(e_3, 0) = (-e_2, 0), \\
J(0, e_2) = (0, e_3),
& & J(0, e_3) = (0, -e_2).
\end{eqnarray*}

A compact Lie group $G := SO(3) \times SO(3)$
acts on $S^2 \times S^2$ transitively and isometrically. 
Then the isotropy subgroup at $o$ is
$$
H = \left\{ \left(
\left[ \begin{array}{cc} 1 & O \\ O & A \end{array} \right],
\left[ \begin{array}{cc} 1 & O \\ O & B \end{array} \right]
\right) \ \bigg| \ A, B \in SO(2) \right\}
\cong SO(2) \times SO(2).
$$
Therefore $S^2 \times S^2$ can be identified with a homogeneous space
$G/H$ in the following manner:
\begin{eqnarray*}
S^2 \times S^2 &\cong& G/H = (SO(3) \times SO(3)) / (SO(2) \times SO(2)), \\
g \cdot o &\longleftrightarrow& gH.
\end{eqnarray*}

We denote the Lie algebras of $G$ and $H$ by $\mathfrak g$ and $\mathfrak h$,
respectively.
Here $\mathfrak g$ and $\mathfrak h$ can be expressed as the following:
\begin{eqnarray*}
\mathfrak g &=& \mathfrak{so}(3) \oplus \mathfrak{so}(3) \\
&=& \left\{ \left(
\left[ \begin{array}{ccc}
0 & -x_1 & -x_2 \\
x_1 & 0 & -x_3 \\
x_2 & x_3 & 0
\end{array} \right],
\left[ \begin{array}{ccc}
0 & -y_1 & -y_2 \\
y_1 & 0 & -y_3 \\
y_2 & y_3 & 0
\end{array} \right]
\right) \ \Bigg| \  x_1, x_2, x_3, y_1, y_2, y_3 \in \mathbb R \right\},\\
\\
\mathfrak h &=& \mathfrak{so}(2) \oplus \mathfrak{so}(2) \\
&=& \left\{ \left(
\left[ \begin{array}{ccc}
0 & 0 & 0 \\
0 & 0 & -z_1 \\
0 & z_1 & 0
\end{array} \right],
\left[ \begin{array}{ccc}
0 & 0 & 0 \\
0 & 0 & -z_2 \\
0 & z_2 & 0
\end{array} \right]
\right) \ \Bigg| \  z_1, z_2 \in \mathbb R \right\}.
\end{eqnarray*}
We set a subspace $\tilde{\mathfrak m}$
of $\mathfrak g$ as
$$
\tilde{\mathfrak m}
= \left\{ \left(
\left[ \begin{array}{ccc}
0 & -x_1 & -x_2 \\
x_1 & 0 & 0 \\
x_2 & 0 & 0
\end{array} \right],
\left[ \begin{array}{ccc}
0 & -y_1 & -y_2 \\
y_1 & 0 & 0 \\
y_2 & 0 & 0
\end{array} \right]
\right) \ \Bigg| \  x_1, x_2, y_1, y_2 \in \mathbb R \right\}.
$$
Then we have a canonical decomposition 
$\mathfrak g = \mathfrak h \oplus \tilde{\mathfrak m}$.
The tangent space $T_o(S^2 \times S^2)$
can be identified with $\tilde{\mathfrak m}$ in a natural manner.
More precisely, the bases of these spaces correspond
with each other as in the following:
\begin{eqnarray*}
T_o(S^2 \times S^2) &\cong& \tilde{\mathfrak m} \\
(e_2, 0) &\longleftrightarrow&
\left( \left[ \begin{array}{ccc}
0 & -1 & 0 \\
1 & 0 & 0 \\
0 & 0 & 0
\end{array} \right],
\left[ \begin{array}{ccc}
0 & 0 & 0 \\
0 & 0 & 0 \\
0 & 0 & 0
\end{array} \right] \right), \\
(e_3, 0) &\longleftrightarrow&
\left( \left[ \begin{array}{ccc}
0 & 0 & -1 \\
0 & 0 & 0 \\
1 & 0 & 0
\end{array} \right],
\left[ \begin{array}{ccc}
0 & 0 & 0 \\
0 & 0 & 0 \\
0 & 0 & 0
\end{array} \right] \right), \\
(0, e_2) &\longleftrightarrow&
\left( \left[ \begin{array}{ccc}
0 & 0 & 0 \\
0 & 0 & 0 \\
0 & 0 & 0
\end{array} \right],
\left[ \begin{array}{ccc}
0 & -1 & 0 \\
1 & 0 & 0 \\
0 & 0 & 0
\end{array} \right] \right), \\
(0, e_3) &\longleftrightarrow& 
\left( \left[ \begin{array}{ccc}
0 & 0 & 0 \\
0 & 0 & 0 \\
0 & 0 & 0
\end{array} \right],
\left[ \begin{array}{ccc}
0 & 0 & -1 \\
0 & 0 & 0 \\
1 & 0 & 0
\end{array} \right] \right).
\end{eqnarray*}

We denote the Grassmannian manifold of all oriented $2$-planes
in $T_o(G/H)$ by $\tilde G_2(T_o(G/H))$.
The action of the rotation group $SO(T_o(G/H)) =: G'$ on $T_o(G/H)$
induces a transitive action of $G'$ on $\tilde G_2(T_o(G/H))$.
We can express $G'$ as a matrix group $SO(4)$
with respect to the basis $\{ (e_2 ,0), (e_3, 0), (0, e_2), (0, e_3) \}$
of $T_o(G/H)$.
Then the isotropy subgroup of the action of $G'$ at
$V_o : = \mathrm{span}_{\mathbb R} \{ (e_2, 0), (e_3, 0) \}
\in \tilde G_2(T_o(G/H))$ is
$$
H' = \left\{ \left[ \begin{array}{cc} A & O \\ O & B \end{array} \right]
\ \bigg| \ A, B \in SO(2) \right\}
\cong SO(2) \times SO(2).
$$
Therefore $\tilde G_2(T_o(G/H))$ can be identified with
a homogeneous space
$$
\tilde G_2(T_o(G/H)) \cong G'/H' = SO(4) / (SO(2) \times SO(2))
$$
with the origin $V_o$.
We also denote Lie algebras of $G'$ and $H'$ by
$\mathfrak g'$ and $\mathfrak h'$,
respectively, i.e.,
\begin{eqnarray*}
\mathfrak g' &=& \mathfrak{so}(4) = \{ X \in M_4(\mathbb R) \ | \ ^tX = -X \}, \\
\mathfrak h' &=& \left\{ \left[ \begin{array}{cc} X & O \\ O & Y \end{array} \right] 
\ \bigg| \ X, Y \in \mathfrak{so}(2) \right\}.
\end{eqnarray*}
Now we set a subspace $\mathfrak m'$ of $\mathfrak g'$ as
$$
\mathfrak m'
:= \left\{ \left[ \begin{array}{cc} O & -^tX \\ X & O \end{array} \right] 
\ \bigg| \ X \in M_2(\mathbb R) \right\}.
$$
Then we have a canonical decomposition
$\mathfrak g' = \mathfrak h' \oplus \mathfrak m'$.
The tangent space of $\tilde G_2(T_o(G/H))$ at $V_0$
can be identified with $\mathfrak m'$.
Take a maximal abelian subspace $\mathfrak a'$ of $\mathfrak m'$ as
$$
\mathfrak a'
:= \left\{ \left[ \begin{array}{cc} O & -^tX \\ X & O \end{array} \right] \ \bigg| \
X = \left[ \begin{array}{cc} \theta_1 & 0 \\ 0 & \theta_2 \end{array} \right] \right\}.
$$
Then the set $\Delta$ of all positive restricted roots of a compact symmetric pair
$(G',H')$ is given by
$$
\Delta = \{ \theta_1 + \theta_2, \theta_1 - \theta_2 \},
$$
and 
$$
\mathcal C := \left\{ \left[ \begin{array}{cc} O & -^tX \\ X & O \end{array} \right] \ \bigg| \
X = \left[ \begin{array}{cc} \theta_1 & 0 \\ 0 & \theta_2 \end{array} \right]
\begin{array}{c} 0 \leq \theta_1 + \theta_2 \leq \pi \\
0 \leq \theta_1 - \theta_2 \leq \pi \end{array} \right\}
$$
is a fundamental domain of $\mathfrak a'$.
The action of $H$ on $\tilde G_2(T_o(G/H)) \cong G'/H'$
is equivalent with the isotropy action of $H'$ on $G'/H'$.
Each orbit of $H$-action on $\tilde G_2(T_o(G/H)) \cong G'/H'$
intersects $\mathrm{Exp}(\mathcal C)$ with only one point.
Hence the orbit space of $H$-action can be identified with $\mathcal C$.
This implies that
$\theta_1 + \theta_2$ and $\theta_1 - \theta_2$ are
invariants of $H$-action on $\tilde G_2(T_o(G/H)) \cong G'/H'$.
Geometrically, $\theta_1 - \theta_2$ is
the K\"ahler angle with respect to a complex structure $J_0 \oplus J_0$.
On the other hand, $\theta_1 + \theta_2$ is the K\"ahler angle
with respect to $J_0 \oplus (-J_0)$.

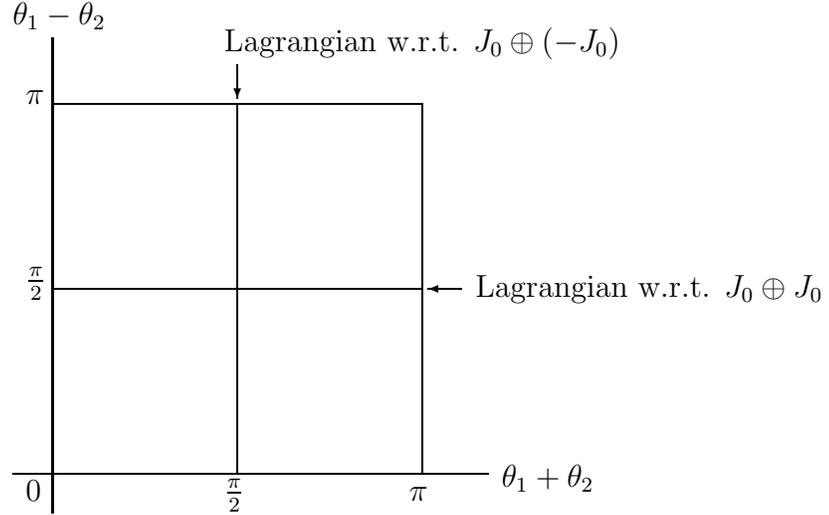
\begin{figure}[h]
\begin{center}
\begin{picture}(200,200)
\put(15,0){\line(0,1){180}}
\put(0,15){\line(1,0){180}}
\put(85,15){\line(0,1){140}}
\put(15,85){\line(1,0){140}}
\put(155,15){\line(0,1){140}}
\put(15,155){\line(1,0){140}}
\put(0,185){$\theta_1 - \theta_2$}
\put(185,10){$\theta_1 + \theta_2$}
\put(5,5){$0$}
\put(80,5){$\frac{\pi}{2}$}
\put(150,5){$\pi$}
\put(5,85){$\frac{\pi}{2}$}
\put(5,155){$\pi$}
\put(85,170){\vector(0,-1){13}}
\put(80,175){Lagrangian w.r.t. $J_0 \oplus (-J_0)$}
\put(170,85){\vector(-1,0){13}}
\put(175,82){Lagrangian w.r.t. $J_0 \oplus J_0$}
\end{picture}
\end{center}
\caption{figure of $\mathcal C$}
\end{figure}

Put $X \in \mathcal C$ as
$$
X = \left[ \begin{array}{cccc}
0 & 0 & -\theta_1 & 0 \\
0 & 0 & 0 & -\theta_2 \\
\theta_1 & 0 & 0 & 0 \\
0 & \theta_2 & 0 & 0
\end{array} \right].
$$
Then $\mathrm{Exp}X \in G'/H' \cong \tilde G_2(T_o(G/H))$
can be expressed as
\begin{eqnarray*}
\mathrm{Exp}X &=& \exp X \cdot V_0 \\
&=& \left[ \begin{array}{cccc}
\cos\theta_1 & 0 & -\sin\theta_1 & 0 \\
0 & \cos\theta_2 & 0 & -\sin\theta_2 \\
\sin\theta_1 & 0 & \cos\theta_1 & 0 \\
0 & \sin\theta_2 & 0 & \cos\theta_2
\end{array} \right] \cdot
\mathrm{span} \{ (e_2, 0), (e_3, 0) \} \\
&=& \mathrm{span} \{ \cos\theta_1 (e_2, 0) + \sin\theta_1 (0, e_2),\
\cos\theta_2 (e_3, 0) + \sin\theta_2 (0, e_3) \}.
\end{eqnarray*}
Hereafter we consider Lagrangian surfaces
with respect to a complex structure $J_0 \oplus J_0$.
Hence we assume that $\theta_1 - \theta_2 = \frac{\pi}{2}$
and put
$$
\theta := \theta_1 = \theta_2 + \frac{\pi}{2} \qquad
\left( \frac{\pi}{4} \leq \theta \leq \frac{3}{4}\pi \right).
$$
Then
$$
\mathrm{Exp}X
= \mathrm{span} \{ \cos\theta (e_2, 0) + \sin\theta (0, e_2),\
\sin\theta (e_3, 0) - \cos\theta (0, e_3) \}.
$$
We denote by $\mathfrak m_\theta$ a $2$-dimensional subspace of
$\tilde{\mathfrak m}$ which corresponds to $\mathrm{Exp}X$
by the identification of $T_o(S^2 \times S^2)$ and $\tilde{\mathfrak m}$.
Then
\begin{eqnarray*}
{\mathfrak m}_\theta
&=& \mathrm{span} \left\{ \left(
\left[ \begin{array}{ccc}
0 & -\cos\theta & 0 \\
\cos\theta & 0 & 0 \\
0 & 0 & 0
\end{array} \right],
\left[ \begin{array}{ccc}
0 & -\sin\theta & 0 \\
\sin\theta & 0 & 0 \\
0 & 0 & 0
\end{array} \right]
\right), \right.\\
& & \hspace{33pt} \left. \left(
\left[ \begin{array}{ccc}
0 & 0 & -\sin\theta \\
0 & 0 & 0 \\
\sin\theta & 0 & 0
\end{array} \right],
\left[ \begin{array}{ccc}
0 & 0 & \cos\theta \\
0 & 0 & 0 \\
-\cos\theta & 0 & 0
\end{array} \right]
\right) \right\} \\
&=& \left\{ 
\left(
\left[ \begin{array}{ccc}
0 & -x_1 \cos\theta & -x_2 \sin\theta \\
x_1 \cos\theta & 0 & 0 \\
x_2 \sin\theta & 0 & 0
\end{array} \right],
\left[ \begin{array}{ccc}
0 & -x_1 \sin\theta & x_2 \cos\theta \\
x_1 \sin\theta & 0 & 0 \\
-x_2 \cos\theta & 0 & 0
\end{array} \right]
\right)
\right. \\
& & \hspace{300pt}\left. \Bigg| \ x_1, x_2 \in \mathbb R \right\}.
\end{eqnarray*}
We have an orthogonal decomposition 
$$
\mathfrak g = \mathfrak h + {\mathfrak m}_\theta + {\mathfrak m}_\theta^\perp.
$$
Now we shall determine $\mathrm{Ker}(\Psi_2|_{\mathfrak h})$
under these notations.
For
\begin{eqnarray*}
Z &=& \left(
\left[ \begin{array}{ccc}
0 & 0 & 0 \\
0 & 0 & -z_1 \\
0 & z_1 & 0
\end{array} \right],
\left[ \begin{array}{ccc}
0 & 0 & 0 \\
0 & 0 & -z_2 \\
0 & z_2 & 0
\end{array} \right]
\right) \in \mathfrak h, \\
X &=& \left(
\left[ \begin{array}{ccc}
0 & -x_1 \cos\theta & -x_2 \sin\theta \\
x_1 \cos\theta & 0 & 0 \\
x_2 \sin\theta & 0 & 0
\end{array} \right],
\left[ \begin{array}{ccc}
0 & -x_1 \sin\theta & x_2 \cos\theta \\
x_1 \sin\theta & 0 & 0 \\
-x_2 \cos\theta & 0 & 0
\end{array} \right]
\right)
\in {\mathfrak m}_\theta,
\end{eqnarray*}
we have
\begin{eqnarray*}
& & \Psi_2(Z)X
= (\mathrm{ad}(Z)X)^\perp \\
&=& \left(
\left[ \begin{array}{ccc}
0 & x_2 z_1 \sin\theta & -x_1 z_1 \cos\theta \\
-x_2 z_1 \sin\theta & 0 & 0 \\
x_1 z_1 \cos\theta & 0 & 0
\end{array} \right],
\left[ \begin{array}{ccc}
0 & -x_2 z_2 \cos\theta & -x_1 z_2 \sin\theta \\
x_2 z_2 \cos\theta & 0 & 0 \\
x_1 z_2 \sin\theta & 0 & 0
\end{array} \right]
\right)^\perp.
\end{eqnarray*}
Thus for
$$
Y = \left(
\left[ \begin{array}{ccc}
0 & -y_1 \sin\theta & -y_2 \cos\theta \\
y_1 \sin\theta & 0 & 0 \\
y_2 \cos\theta & 0 & 0
\end{array} \right],
\left[ \begin{array}{ccc}
0 & y_1 \cos\theta & -y_2 \sin\theta \\
-y_1 \cos\theta & 0 & 0 \\
y_2 \sin\theta & 0 & 0
\end{array} \right]
\right)
\in {\mathfrak m}_\theta^\perp,
$$
we have
\begin{eqnarray*}
\langle \Psi_2(Z)X, Y \rangle
&=& \langle [Z, X]^\perp, Y \rangle
= \langle[Z, X], Y \rangle
= -\frac{1}{2} \mathrm{tr}(ZXY - XZY) \\
&=& x_1 y_2 (z_1 \cos^2\theta + z_2 \sin^2\theta)
- x_2 y_1 (z_1 \sin^2\theta + z_2 \cos^2\theta).
\end{eqnarray*}
Note that $Z \in \mathfrak h$ is in $\mathrm{Ker}(\Psi_2|_{\mathfrak h})$
if and only if $\langle \Psi_2(Z)X, Y \rangle = 0$
for any $X \in {\mathfrak m}_\theta, Y \in {\mathfrak m}_\theta^\perp$.
Thus
\begin{eqnarray*}
Z \in \mathrm{Ker}(\Psi_2|_{\mathfrak h})
&\Longleftrightarrow& \langle \Psi_2(Z)X, Y \rangle = 0 \qquad
(^\forall X \in {\mathfrak m}_\theta,\ ^\forall Y \in {\mathfrak m}_\theta^\perp) \\
&\Longleftrightarrow& \left\{ \begin{array}{l}
z_1 \cos^2\theta + z_2 \sin^2\theta = 0\\
z_1 \sin^2\theta + z_2 \cos^2\theta = 0
\end{array} \right. \\
&\Longleftrightarrow& \left\{ \begin{array}{l}
z_1 = -z_2 \qquad (\text{if}\ \cos^2\theta = \sin^2\theta)\\
z_1 = z_2 = 0 \qquad (\text{if}\ \cos^2\theta \neq \sin^2\theta).
\end{array} \right.
\end{eqnarray*}
This yields that when $\theta = \frac{\pi}{4}$ or $\theta = \frac{3}{4}\pi$
$$
\mathrm{Ker}(\Psi_2|_{\mathfrak h})
= \left\{ \left(
\left[ \begin{array}{ccc}
0 & 0 & 0 \\
0 & 0 & -z \\
0 & z & 0
\end{array} \right],
\left[ \begin{array}{ccc}
0 & 0 & 0 \\
0 & 0 & z \\
0 & -z & 0
\end{array} \right]
\right)
\ \Bigg| \ z \in \mathbb R
\right\},
$$
hence $\dim (\mathrm{Im}(\Psi_2|_{\mathfrak h})) = 1$.
Otherwise $\mathrm{Ker}(\Psi_2|_{\mathfrak h}) = \{ 0 \}$, hence
$\dim (\mathrm{Im}(\Psi_2|_{\mathfrak h})) = 2$.

Let $L$ be a Lagrangian surface of $S^2 \times S^2$ with respect to
a complex structure $J_0 \oplus J_0$.
Assume that $L$ contains $o = (e_1, e_1) \in S^2 \times S^2$.
Then from Theorem \ref{thm:Killing nullity}
we have
\begin{equation} \label{eq:5-1}
\mathrm{nul}_K(L)
\geq \mathrm{codim}(L) + \dim (\mathrm{Im}(\Psi_2|_{\mathfrak h}))
\geq 3.
\end{equation}
Since $\dim (\mathrm{Im}(\Psi_2|_{\mathfrak h}))$ is invariant 
under the action of $H$,
the equality of the second inequality of (\ref{eq:5-1}) holds
if and only if
$T_o(L)$ is contained in subset $H \cdot {\mathfrak m}_{\pi/4}$
or $H \cdot {\mathfrak m}_{3\pi/4}$
of $\tilde G_2(T_o(G/H))$.
Let $\tilde G_2(T(G/H))$ denote the Grassmannian bundle 
over $G/H$ whose fiber at each point $p \in G/H$ is $\tilde G_2(T_p(G/H))$.
Since any point of $L$ can be moved to the origin $o$ by the action of $G$,
we have the following proposition.

\begin{pro} \label{pro:inequality of the Killing nullity}
Let $L$ be a Lagrangian surface with respect to a complex structure
$J_0 \oplus J_0$ on $S^2 \times S^2$.
For any $p \in L$, take $g \in G$ such that $gp=o$.
Then the Killing nullity of $L$ satisfies the inequality
\begin{equation} \label{eq:inequality of the Killing nullity}
\mathrm{nul}_K(L)
= \mathrm{nul}_K(gL)
\geq \mathrm{codim}(gL) + \dim \mathrm{Im}(\Psi_2|_{\mathfrak h})
\geq 3.
\end{equation}
Moreover, the equality condition of the last inequality in
(\ref{eq:inequality of the Killing nullity}) holds
for all $p \in L$ if and only if
the tangent bundle $TL$ of $L$ is contained in the subbundle
$G \cdot {\mathfrak m}_{\pi/4}$ or $G \cdot {\mathfrak m}_{3\pi/4}$
of $\tilde G_2(T(G/H))$.
\end{pro}

Now we study the condition that the equality of the last inequality of
(\ref{eq:inequality of the Killing nullity}) will be satisfied.
When $\theta = \frac{\pi}{4}$,
we have $\theta_1 + \theta_2 = 0,\ \theta_1 - \theta_2 = \frac{\pi}{2}$.
When $\theta = \frac{3}{4}\pi$,
we have $\theta_1 + \theta_2 = \pi,\ \theta_1 - \theta_2 = \frac{\pi}{2}$.
Hence ${\mathfrak m}_{\frac{\pi}{4}}$ and ${\mathfrak m}_{\frac{3}{4}\pi}$
are Lagrangian subspaces of $\tilde{\mathfrak m} \cong T_o(G/H)$
with respect to $J_0 \oplus J_0$,
and are complex subspaces with respect to $J_0 \oplus (-J_0)$.
Therefore the equality of the last inequality
of (\ref{eq:inequality of the Killing nullity}) holds for all $p \in L$
if and only if $L$ is a complex submanifold of $S^2 \times S^2$
with respect to a complex structure $J_0 \oplus (-J_0)$.
A complex submanifold of a K\"ahler manifold is a calibrated submanifold,
so it is volume minimizing in its homology class,
in particular it is a stable minimal submanifold.
Castro and Urbano \cite{CU} obtained the following result
for stable minimal Lagrangian surfaces in $S^2 \times S^2$.

\begin{thm}[Castro-Urbano \cite{CU}]
The only stable compact minimal Lagrangian surface of $S^2 \times S^2$
is the totally geodesic Lagrangian sphere
$$
{\bf M}_0 := \{ (x,-x) \in S^2 \times S^2 \ | \ x \in S^2 \}.
$$
\end{thm}

Hence, we have

\begin{cor} \label{cor:5-3}
Let $L$ be a compact connected Lagrangian surface in $S^2 \times S^2$
with respect to a complex structure $J_0 \oplus J_0$.
When we move any point of $L$ to the origin $o$,
the inequality
$$
\mathrm{nul}_K(L)
\geq \mathrm{codim}(L) + \dim \mathrm{Im}(\Psi_2|_{\mathfrak h})
\geq 3
$$
is satisfied.
Moreover, the equality of the second inequality in the above formula
holds for all points of $L$ if and only if $L$ is congruent to ${\bf M}_0$.
\end{cor}

\section{Classification of Lagrangian surfaces with low\\ Killing nullities}
\setcounter{equation}{0}

In this section, using the inequality in
Proposition \ref{pro:inequality of the Killing nullity},
let us classify Lagrangian surfaces of $S^2 \times S^2$ with
low Killing nullities.

\subsection{The case where $\mathrm{nul}_K(L) = 3$}

Let $L$ be a compact connected Lagrangian surface in $S^2 \times S^2$.
Assume that $\mathrm{nul}_K(L) = 3$.
Then, by Corollary \ref{cor:5-3}, we have
$$
3 = \mathrm{nul}_K(L)
\geq \mathrm{codim}(L) + \dim \mathrm{Im}(\Psi_2|_{\mathfrak h})
\geq 3
$$
for all points of $L$.
Hence, the equality condition of the second inequality holds.
Using Corollary \ref{cor:5-3} again, $L$ must be congruent to
the totally geodesic Lagrangian sphere ${\bf M}_0$.
We can prove that ${\bf M}_0$ is globally tight (see Section 7).

\subsection{The case where $\mathrm{nul}_K(L) = 4$}

Next, assume that $\mathrm{nul}_K(L) = 4$.
Then $L$ cannot be congruent to ${\bf M}_0$.
Therefore, by Corollary \ref{cor:5-3}, there exist $p \in L$ and $g \in G$
such that $gp=o$ and
$$
4 = \mathrm{nul}_K(L)
= \mathrm{nul}_K(gL)
\geq \mathrm{codim}(gL) + \dim \mathrm{Im}(\Psi_2|_{\mathfrak h})
\geq 4.
$$
Hence, the equality condition of the first inequality holds.
By Theorem \ref{thm:Killing nullity},
$L$ is a homogeneous Lagrangian surface in $S^2 \times S^2$.
Here, let us use the following recent result on homogeneous
Lagrangian surfaces in $S^2 \times S^2$.

\begin{thm}[Ma-Ohnita \cite{MO}] \label{thm:Ma-Ohnita}
Let $L$ be a compact homogeneous Lagrangian surface in $S^2 \times S^2$.
Then $L$ must be congruent to either the totally geodesic Lagrangian sphere
$$
{\bf M}_0 = \{ (x, -x) \in S^2 \times S^2 \ | \ x \in S^2 \}
$$
or Lagrangian tori obtained by a product of latitude circles in $S^2$
$$
T_{a,b} := \{ (x, y) \in S^2 \times S^2 \ | \ x_1 = a,\ y_1 = b \} \qquad
(0 \leq a, b < 1).
$$
\end{thm}

Note that $\mathrm{nul}_K({\bf M}_0) = 3,\ \mathrm{nul}_K(T_{a,b}) = 4$.
Therefore, Theorem \ref{thm:Ma-Ohnita} implies that 
the Lagrangian surface $L$ must be congruent to $T_{a,b}$.
It is clear that $T_{a,b}$ is tight and, especially, 
the totally geodesic Lagrangian torus ${\bf T} := T_{0,0}$ is globally tight.
\footnote{The symbols ${\bf M}_0, T_{a,b}$ and ${\bf T}$
were introduced in Castro and Urbano's paper \cite{CU}.}

\bigskip

Thus we finish the proof of Theorem \ref{thm:1-2}.

\section{Global tightness of the Lagrangian surface ${\bf M}_0$}

In this section, we give a proof of the following theorem.
This completes the proof of Corollary \ref{cor:1-3}.

\begin{thm} \label{thm:main}
The totally geodesic Lagrangian
sphere ${\bf M}_0$ in $S^2 \times S^2$ is globally tight.
\end{thm}

\begin{rem} \rm
The fact that the totally geodesic Lagrangian submanifold
$\mathbb RP^n \subset \mathbb CP^n$ is globally tight
has been proven by Howard \cite{Howard} using a different method.
\end{rem}

First, we shall review the generalized Poincar\'e formula in Riemannian
homogeneous spaces obtained by Howard \cite{Howard}.

Let $U$ be a finite dimensional real vector space with an inner product,
and $V$ and $W$ vector subspaces in $U$.
Take orthonormal bases $v_1, \cdots, v_p$ of $V$ and $w_1, \cdots, w_q$
of $W$.
The angle $\sigma(V,W)$ between $V$ and $W$ is defined by
$$
\sigma(V,W)
= \| v_1 \wedge \cdots \wedge v_p \wedge w_1 \wedge \cdots \wedge w_q \|.
$$

Let $G$ be a Lie group equipped with a left invariant Riemannian metric
and $K$ a closed subgroup of $G$.
Moreover, we assume that the metric on $G$ is biinvariant on $K$.
Then, for a subspace $V$ of $T_x(G/K)$ and a subspace $W$ of $T_y(G/K)$,
we take $g_x, g_y \in G$ satisfying $g_x o = x$ and $g_y o = y$.
We define the angle $\sigma_K(V,W)$ between $V$ and $W$ by
\begin{equation}\label{eq:7-1}
\sigma_K(V,W)
= \int_K \sigma( (dg_x)_o^{-1} V, (dk)_o^{-1} (dg_y)_o^{-1} W ) d\mu(k).
\end{equation}

\begin{thm}[Howard \cite{Howard}, Poincar\'{e} formula] \label{th:H}
Let $G/K$ be a Riemannian homogeneous space and assume that $G$ is unimodular.
Let $M$ and $N$ be submanifolds of $G/K$ with $\dim (G/K) \leq \dim M + \dim N$.
Then we have
$$
\int_G \mbox{vol}(M \cap gN) d\mu(g)
= \int_{M \times N} \sigma_K(T_x^\perp M, T_y^\perp N) d\mu(x,y).
$$
\end{thm}

Here we apply Theorem \ref{th:H} in the case of $S^2 \times S^2$ and
calculate the integration of intersection numbers $\#(M \cap gN)$
when $M=N={\bf M}_0$.

Let us put
$o := (e_1, -e_1) \in S^2 \times S^2 \subset \mathbb R^3 \times \mathbb R^3$.
Note that $o \in {\bf M}_0$.
The tangent space of $S^2 \times S^2$ at $o$ is given by
$$
T_o(G/K) = T_o(S^2 \times S^2) =T_{e_1}(S^2) \oplus T_{-e_1}(S^2)
$$
and $\{ (e_2, 0), (e_3, 0), (0, e_2), (0, e_3) \}$ forms an orthonormal
basis of $T_o(G/K)$.
Moreover, since ${\bf M}_0$ is a homogeneous submanifold,
for any $\mathbf{x} \in {\bf M}_0$, there exists $g \in G$ such that
$$
u_1 := \frac{1}{\sqrt{2}}(e_2, -e_2), \qquad
u_2 := \frac{1}{\sqrt{2}}(e_3, -e_3)
$$
is an orthonormal basis of $(dg)^{-1}_o(T_{\mathbf x}{\bf M}_0)$ and
$$
v_1 := \frac{1}{\sqrt{2}}(e_2, e_2), \qquad
v_2 := \frac{1}{\sqrt{2}}(e_3, e_3)
$$
is an orthonormal basis of $(dg)^{-1}_o(T_{\mathbf x}^\perp{\bf M}_0)$.

Then, from (\ref{eq:7-1}), we have
$$
\sigma_K (T_{\mathbf x}^\perp {\bf M}_0, T_{\mathbf y}^\perp {\bf M}_0)
= \int_K \| v_1 \wedge v_2 \wedge k^{-1}(v_1 \wedge v_2) \| d\mu(k).
$$
By the Hodge $*$-operator, we have
$$
\sigma_K (T_{\mathbf x}^\perp {\bf M}_0, T_{\mathbf y}^\perp {\bf M}_0)
= \int_K | \langle u_1 \wedge u_2, k^{-1}(v_1 \wedge v_2) \rangle | d\mu(k).
$$
Since $K = SO(2) \times SO(2)$, we can put
$$
a = \left( \left[ \begin{array}{cc} \cos\phi & -\sin\phi \\
\sin\phi & \cos\phi \end{array} \right],
\left[ \begin{array}{cc} 1 & 0 \\
0 & 1 \end{array} \right] \right),
\quad
b = \left( \left[ \begin{array}{cc} 1 & 0 \\
0& 1 \end{array} \right],
\left[ \begin{array}{cc} \cos\psi & -\sin\psi \\
\sin\psi & \cos\psi \end{array} \right] \right)
$$
and $k = b^{-1}a$. Then
$$
\sigma_K (T_{\mathbf x}^\perp {\bf M}_0, T_{\mathbf y}^\perp {\bf M}_0)
= \int_0^{2 \pi} \int_0^{2 \pi}
| \langle a(u_1 \wedge u_2), b(v_1 \wedge v_2) \rangle | d\phi\, d\psi.
$$
Here, by a direct calculation, we obtain
$$
\langle a(u_1 \wedge u_2), b(v_1 \wedge v_2) \rangle
= \frac{1}{2} (1 - \cos(\phi + \psi)).
$$
Hence we have
$$
\sigma_K (T_{\mathbf x}^\perp {\bf M}_0, T_{\mathbf y}^\perp {\bf M}_0)
= \frac{1}{2} \int_0^{2 \pi} \int_0^{2 \pi}
| 1 - \cos(\phi + \psi) | d\phi\, d\psi
= 2 \pi^2.
$$
Therefore Theorem \ref{th:H} yields
$$
\int_G \#({\bf M}_0 \cap g{\bf M}_0) d\mu(g)
= \int_{{\bf M}_0 \times {\bf M}_0} \sigma_K(T_{\mathbf x}^\perp {\bf M}_0,
T_{\mathbf y}^\perp {\bf M}_0) d\mu(x, y)
= 2 \pi^2 \mathrm{vol}({\bf M}_0)^2.
$$
Since
$$
\mathrm{vol}({\bf M}_0) = 2 \mathrm{vol}(S^2(1)) = 2 \cdot 4 \pi = 8 \pi,
$$
we have
\begin{equation} \label{eq:7-3}
\int_G \#({\bf M}_0 \cap g{\bf M}_0) d\mu(g) = 128 \pi^4.
\end{equation}

Here we review the Arnold-Givental inequality for real forms in Hermitian
symmetric spaces.

\begin{thm}[Oh \cite{Oh93}, \cite{Oh95}, \cite{Oh952}] \label{thm:Oh}
Let $G/K$ be a compact Hermitian symmetric space and $L$ be a real form
of $G/K$.
Assume that the minimal Maslov number of $L$ is greater than or equal to 2.
Then for any Hamiltonian diffeomorphism $\rho \in \mathrm{Ham}(G/K)$ of $G/K$
such that $L$ and $\rho(L)$ intersect transversely, the inequality
$$
\#(L \cap \rho (L)) \geq
SB(L, \mathbb Z_2)
$$
holds.
\end{thm}

Since the minimal Maslov number of ${\bf M}_0 \subset S^2 \times S^2$ is
greater than or equal to 2, the assumption of the above theorem is satisfied.

\begin{proof}[Proof of Theorem \ref{thm:main}]
Assume that ${\bf M}_0 \subset S^2 \times S^2$ is not
globally tight.
Then, there exists $g_0 \in G$ such that ${\bf M}_0$ and $g_0{\bf M}_0$
intersect transversely and
$$
\#({\bf M}_0 \cap g_0{\bf M}_0) \geq SB({\bf M}_0, \mathbb Z_2) + 1.
$$
Then, there exists an open neighborhood $U$ of $g_0$ in $G$ satisfying
$$
\#({\bf M}_0 \cap g{\bf M}_0) \geq SB({\bf M}_0, \mathbb Z_2) + 1
$$
for all $g \in U$.
By equality (\ref{eq:7-3}) and Theorem \ref{thm:Oh}, we have
\begin{eqnarray*}
128 \pi^4 &=& \int_G \#({\bf M}_0 \cap g{\bf M}_0) d\mu(g) \\
&=& \int_{G \backslash U} \#({\bf M}_0 \cap g{\bf M}_0) d\mu(g)
+ \int_U \#({\bf M}_0 \cap g{\bf M}_0) d\mu(g) \\
& \geq & \int_{G \backslash U} SB({\bf M}_0, \mathbb Z_2) d\mu(g)
+ \int_U (SB({\bf M}_0, \mathbb Z_2) + 1) d\mu(g) \\
&=& \int_G SB({\bf M}_0, \mathbb Z_2) d\mu(g) + \int_U d\mu(g) \\
&>& SB({\bf M}_0, \mathbb Z_2) \mathrm{vol}(G)
\end{eqnarray*}
Since $SB({\bf M}_0, \mathbb Z_2) = SB(S^2, \mathbb Z_2) = 2$ and
$$
\mathrm{vol}(G) = \mathrm{vol}(SO(3))^2 = (8 \pi^2)^2 = 64 \pi^4,
$$
we have
$$
128 \pi^4 > SB({\bf M}_0, \mathbb Z_2) \mathrm{vol}(G) = 128 \pi^4.
$$
This is a contradiction.
Therefore, ${\bf M}_0 \subset S^2 \times S^2$ is globally tight.
\end{proof}

\section*{Acknowledgements}

We would like to thank Professor Yoshihiro Ohnita for informing us
about his recent work \cite{MO} on the classification of
homogeneous Lagrangian submanifolds in complex hyperquadrics,
which is used in \S 6.2.

\noindent
Hiroshi Iriyeh\\
{\sc School of Science and Technology for Future Life\\
Tokyo Denki University\\
2-2 Kanda-Nishiki-Cho, Chiyoda-Ku\\
Tokyo, 101-8457, Japan}

{\it E-mail address} : {\tt hirie@im.dendai.ac.jp}

\vspace{5mm}

\noindent
Takashi Sakai\\
{\sc Graduate School of Science\\
Osaka City University\\
3-3-138, Sugimoto, Sumiyoshi-ku\\
Osaka, 558-8585, Japan}

{\it E-mail address} : {\tt tsakai@sci.osaka-cu.ac.jp}

\end{document}